\documentclass[11pt]{article}
\usepackage{stmaryrd} 
\usepackage{latexsym}
\usepackage{amsmath,amsthm,amsfonts,amssymb, mathrsfs, wasysym, enumitem} 
\usepackage{graphicx}
\usepackage{color}

\newtheorem{thm}{Theorem}[section]
\newtheorem{theo}[thm]{Theorem}

\newtheorem{prop}[thm]{Proposition} 
 
\newtheorem{lem}[thm]{Lemma} 
\newtheorem{lemma}[thm]{Lemma}

\makeatletter
\@tfor\next:=abcdefghijklmnopqrstuvwxyzABCDEFGHIJKLMNOPQRSTUVWXYZ\do{%
  \def\command@factory#1{%
    \expandafter\def\csname cal#1\endcsname{\mathcal{#1}}
    \expandafter\def\csname frak#1\endcsname{\mathfrak{#1}}
    \expandafter\def\csname scr#1\endcsname{\mathscr{#1}}
    \expandafter\def\csname bb#1\endcsname{\mathbb{#1}}
    \expandafter\def\csname rm#1\endcsname{\mathrm{#1}}
  }
 \expandafter\command@factory\next
}
\makeatother

\newcommand{\Fix} {\operatorname{Fix}}

\newcommand{\diam}{\mathop{\mathrm{diam}\;}}

\title{Property $P_{naive}$ for acylindrically hyperbolic groups}
\author{Carolyn R. Abbott, Fran\c{c}ois Dahmani} 
\date{}
\begin{document}

\maketitle

\begin{abstract} We prove that every acylindrically hyperbolic group that has
no non-trivial finite normal subgroup satisfies a strong ping pong
property, the $P_{naive}$ property: for any finite collection of elements
$h_1, \dots, h_k$, there exists another element $\gamma\neq 1$ such that for
all $i$, $\langle h_i, \gamma
\rangle = \langle h_i \rangle* \langle \gamma \rangle$.  We also
show that if a collection of subgroups $H_1, \dots, H_k$  is a
hyperbolically embedded collection, then  there is  $\gamma \neq 1$ such that for
all $i$, $\langle H_i, \gamma
\rangle = H_i * \langle \gamma \rangle$. \end{abstract}

The Ping-Pong lemma and its many variations are iconic arguments in group theory that are particularly useful when dealing with groups acting on hyperbolic spaces.  They allow one to produce free subgroups at
will in many groups.  

One of the  strongest ping-pong properties of a group is property $P_{naive}$,
defined in \cite{BCdlH}. $P_{naive}$ stands for ``naively permissive:" a group satisfying it {\it 
``(...) is so free that, for any finite subset $F$ of $G\setminus\{1\}$ there
exists a partner $\gamma_0$ of infinite order in $G$ such that each
pair $\{x, \gamma_0\}$ with $x\in F$ generates a subgroup which is as
free as possible." (\cite{BCdlH})}. Here, ``as free as possible'' means
that the subgroup generated is canonically isomorphic to $\langle x\rangle_G *
\langle \gamma_0\rangle_G \simeq \langle x\rangle_G *\mathbb{Z}  $. 

It is known that Zariski-dense discrete subgroups of connected simple groups of
rank $1$ with trivial center have property $P_{naive}$ \cite[Thm 3]{BCdlH}, as well as hyperbolic groups
and relatively hyperbolic groups without non-trivial finite normal
subgroups \cite{AM} (see also \cite{dlH} for torsion-free hyperbolic
groups), and
non-Euclidean cubical $\operatorname{CAT}(0)$ groups without non-trivial finite normal subgroups \cite{KS}.

The action of a group $G$ on a metric space $X$ is \emph{acylindrical}
if for all $\epsilon>0$ there exist constants $M,N\geq 0$ such that
for all $x,y\in X$ with $d(x,y)\geq M$, the number of elements $g\in
G$ satisfying $d(x,gx)\leq \epsilon$ and $d(y,gy)\leq \epsilon$ is at
most $N$.  One says that a group is {\it acylindrically hyperbolic} if
it admits a non-elementary, acylindrical, cobounded action
by isometries on a hyperbolic space.  The class of acylindrical hyperbolic groups, whose systematic study
was initiated 
 by Osin \cite{O}
after  studies in \cite{DGO, H, BBF}, has attracted much interest due to its
many examples.

 In this note, we show that all acylindrically hyperbolic groups
 without non-trivial finite normal subgroups indeed  have property
 $P_{naive}$ (Theorem \ref{theo;main}). This includes hyperbolic groups, relatively hyperbolic groups, non-virtually cyclic $CAT(0)$ groups with a rank one isometry, non-exceptional mapping class groups, $\operatorname{Out}(F_n)$ for $n\geq 2$, certain
 models of infinitely presented random groups, and many other
 examples.  In particular, we recover and greatly extend one of the main
 results  of  \cite{AM} and  \cite{KS}.

We actually prove a stronger statement (Theorem \ref{theo;pingpong_with_elliptic}), of which property $P_{naive}$ is a consequence: 

\begin{thm} \label{main1}
Given a group $G$ without non-trivial finite normal subgroups that admits a non-elementary cobounded  acylindrical action by isometries on a hyperbolic space $X$ and a collection of subgroups $H_1,\dots,H_n$ which are elliptic in $X$, there exists an element $\gamma\in G$ such that $\langle \gamma ,H_i\rangle \simeq \langle \gamma\rangle *H_i$.
\end{thm}

 Recall that given a group $G$ acting by isometries on a hyperbolic space $X$, an
element $g\in G$ is called {\em elliptic} if the orbit $g\cdot s$ is
bounded for some (equivalently any) $s\in X$, and it is called {\em
  loxodromic } if the map $\mathbb Z\to X$ defined by $n\mapsto g^ns$
is a quasi-isometric embedding for some (equivalently any) $s\in X$. A
subgroup is called elliptic if all its elements are elliptic.

To prove this statement, we find a loxodromic element in $G$  whose axis
escapes all quasi-fixed-point sets for these subgroups.   We then show
that a sufficiently large power of this loxodromic element plays ping-pong with each of the given
subgroups.   Thus the proof of property $P_{naive}$ for acylindrically hyperbolic groups reduces to finding a hyperbolic space with an acylindrical action of $G$  in which all of
the given elements are elliptic.

It is known that if a group has Property $P_{naive}$ then its reduced $C^*$--algebra is simple. For acylindrically hyperbolic groups, this simplicity was shown in
\cite[Thm 2.35]{DGO}, but it was established through other means.   Thus Theorem \ref{theo;main} provides a new proof of $C^*$--simplicity for acylindrically hyperbolic groups without non-trivial finite normal subgroups.

As a consequence of Theorem \ref{main1}, we obtain the following strengthening of \cite[Thm 2.35]{DGO}.
  
  \begin{thm} Suppose $G$ is an acylindrically hyperbolic group.  Then the following are equivalent.
  \begin{enumerate}
  \item $G$ has no non-trivial finite normal subgroups.
  \item $G$ is ICC.
  \item $G$ is not inner amenable.
  \item $G$ has property $P_{naive}$.
  \end{enumerate}
  If, in addition, $G$ is countable, the above conditions are also equivalent to 
  \begin{enumerate}[resume]
  \item The reduced $C^*$--algebra of $G$ is simple.
  \item The reduced $C^*$--algebra of $G$ has a unique normalized trace.

  \end{enumerate}
  
  \end{thm}
  
Recently, Kalantar and Kennedy have given a dynamical criterion for $C^*$--simplicity:  they show in \cite{KK} that a group $G$ is $C^*$--simple if and only if $G$ has a strongly proximal, topologically free, minimal action by homeomorphisms on a compact topological space (see Section 3 for precise definitions).    Our proof of Theorem \ref{main1} leads to the following result (Proposition \ref{prop:boundaryaction}), where $\partial X$ is the Gromov boundary of the hyperbolic space $X$.  
 
 \begin{prop} \label{main2}
Let $G$ be an acylindrically hyperbolic group without non-trivial finite normal subgroup, and let $X$ be a hyperbolic space with an isometric $G$-action which is acylindrical and cobounded.

The action of $G$ on   $\partial X$ by homeomorphisms is minimal, strongly proximal, and topologically free.
\end{prop}

However, in general $\partial X$ is not compact, and thus this does
not give another independent proof of $C^*$--simplicity.  In \cite{KS}, Kar and
Sageev prove that in the case of a cubical CAT(0) group, a compact
subset of the roller boundary is a topologically free G-boundary.
Proposition 0.3 is close to this, which possibly suggests that in some
cases, compactness is not crucial, and that a suitable bordification
is enough.

We now discuss another consequence of property $P_{naive}$.  A ring $R$ is \emph{(right) primitive} if it has a faithful irreducible (right) $R$--module. For many years, it was not known whether there exists a group $G$ and a ring $R$ such that the group ring $RG$ is primitive, and it was conjectured that no such groups exist.  However, there are actually many groups for which this is true; Formanek and Snider give the first example of such a group in \cite {FS}.  Formanek showed in \cite{Form} that under mild hypotheses, the group algebra of the free product of groups over any field  is primitive, and Balogun later generalized this result to amalgamated products in \cite{Bal}.  Recently, Solie \cite{Sol} showed that the group ring of a non-elementary torsion-free hyperbolic group over any countable domain is primitive.  Using Theorem \ref{main1}, we can generalize these results to acylindrically hyperbolic groups.  We note that this result is new even for hyperbolic groups with torsion.

\begin{thm}
Let $G$ be a countable acylindrically hyperbolic group with no non-trivial finite normal subgroups.  For any countable domain $R$, the group ring $RG$ is primitive.
\end{thm}

\subsection*{Acknowledgments}

The authors thank M. Sageev for asking the question and encouraging us to write
this note. The authors also thank D. Osin for pointing out the application to group rings, as well as the anonymous referee for useful comments.
C.A. is partially supported by NSF RTG award DMS-1502553. 
F.D. is supported by the Institut Universitaire de France. 
The authors would like to thank the Mathematical Sciences Research Institute for hosting
them during the completion of this project.

\section{Axes of loxodromic elements}

In this section we collect some useful facts regarding the axes of
loxodromic elements in acylindrical actions on hyperbolic spaces. 

Given an isometry $\phi$ of a space $X$, we say that a point is
\emph{$r$--quasi-minimal for $\phi$} if $d(x,\phi(x))\leq \inf_{y\in X} d(y, \phi(y)) +r$.  The set consisting of all $r$--quasi-minimal points for $\phi$ is an \emph{$r$--quasi-axis} for $\phi$, which we denote $A_r(\phi)$.  If the constant $r$ is unimportant or clear from  context, we may  call $A_r(\phi)$ simply a \emph{quasi-axis} for $\phi$.  If $\phi$ is a loxodromic isometry on a hyperbolic space, then $\phi$ fixes two points $\phi^+$ and $\phi^-$ on the boundary, where $\phi^+$ is the attracting fixed point of $\phi$.


Let $G$ be a acylindrically hyperbolic group 
and consider an acylindrical action of $G$ on a $\delta$--hyperbolic geodesic space $X$.  For a subgroup $H$ of $G$, we denote by $\Fix_K(H)$ the set $\{ x\in X\mid d(x,hx)\leq
K \textrm{ for all } h\in H\}$.  Given $g,h\in G$, we use the conjugation notation $g^h=hgh^{-1}$.

\begin{lemma} \label{lem:axis}  Let $G$ act acylindrically on a $\delta$--hyperbolic geodesic metric space $(X,d)$, and suppose $\gamma\in G$ is loxodromic for this action.  

 If $\gamma^-\in \overline{\Fix_{50\delta}(H)}$ for some elliptic subgroup $H\subset G$, then $\gamma^+\in\overline{\Fix_{50\delta}(H)}$ as well.   
 Similarly, if $\gamma^+\in\overline{\Fix_{50\delta}(H)}$, then $\gamma^-\in \overline{\Fix_{50\delta}(H)}.$
\end{lemma}

\begin{proof}
Let $A(\gamma)=A_{10\delta}(\gamma)$.  By acylindricity of the action, if $H$ is an infinite subgroup, then $\Fix_{50\delta}(H)$ is bounded.  Thus if $\gamma^-\in\overline{\Fix_{50\delta}(H)}$, it must be that $|H|<\infty$.  As $\Fix_{50\delta}(H)$ is $K$--quasi-convex for some $K$, there is a quasi-ray $R$ asymptotic to $A(\gamma)$ with $R^-=\gamma^-$ such that $R\subset \Fix_{50\delta}(H)$.  Let $T$ be the union of all such quasi-rays.  Then there exists a constant $D=D(K,\delta)$ such that for each $h\in H$ and all $x\in T$, we have $d(hx,x)\leq D$.  

As $\gamma^{-N}x\in T$ for $N\geq 0$, it follows that for each $h\in
H$, $$d(x,\gamma^N h\gamma^{-N}x)=d(\gamma^{-N}x,h\gamma^{-N}x)\leq
D.$$  As $T$ is a union of infinite rays, we can choose a point $y\in T$ that
is arbitrarily far from $x$.  For all $N$, $\gamma^{N}h\gamma^{-N}$
moves both $x$ and $y$ by at most $D$.  However, by the acylindricity
of the action the number of such elements is uniformly bounded.  Thus
for each $h\in H$ there is a constant $N_0\geq 0$ such that
$\gamma^{N_0}h\gamma^{-N_0}=h$. 
  Hence, $h\in C_G(\gamma^{N_0})$ and so $d_{\textrm{Haus}}(A(\gamma),hA(\gamma))<\infty$.  Therefore, $h\gamma^+=\gamma^+$ for each $h\in H$.

We will now show that $\gamma^+\in\overline{\Fix_{50\delta}(h)}$ for each $h\in H$.  It will then follow that $\gamma^+\in\overline{\Fix_{50\delta}(H)}$.  Suppose towards a contradiction that $\gamma^+\not\in\overline{\Fix_{50\delta}(h)}$ for some $h\in H$, and consider a sequence $(x_i)_{i\in\mathbb N}$ converging to $\gamma^+$.  It follows that $(hx_i)_{i\in\mathbb N}$ converges to $h\gamma^+$.  By assumption, there exists some $K>0$ such that for all $i\geq K$, $x_i\not\in\Fix_{50\delta}(h)$, so for all $i\geq K$, $d(x_i,hx_i)>50\delta$. Therefore the Gromov product $(\gamma^+\cdot hx_n)_{x_0}\leq 100\delta$, and so $\gamma^+\neq h\gamma^+$, a contradiction.  

\end{proof}


\begin{lemma}(Location of the axis)  \label{LOA}
  Let $X$ be a geodesic $\delta$--hyperbolic space, and
   $\gamma_1,\gamma_2$ two loxodromic isometries of $X$ with disjoint
   pairs of fixed points in $\partial X$. 
Let $V$ be a neighborhood of $\gamma_1^+$ and $U$ be a neighborhood of
$\gamma_2^-$. Then, for sufficiently large $n$ and $k$, $\gamma_1^n\gamma_2^k$ is a
loxodromic isometry whose repelling fixed point in is $U$ and whose
attracting fixed point is in $V$.

\end{lemma}
\begin{proof}

For $i=1,2$, let $A_i=A_{10\delta}(\gamma_i)$ be a $10\delta$--quasi-axis for $\gamma_i$.  
Let $x$ be an arbitrary point in the closest point projection of $A_2$ on $A_1$, and let $y$ be an arbitrary point in the  
 closest point projection of $x$ on $A_2$. 
Let $\sigma_0$ be a sub-segment of $A_2$ of length $1000\delta$ starting
at $y$, and let $y_0$ be its terminal point 
(see Figure \ref{fig1}).  

\begin{figure}[h]
\def\svgwidth{2in}  
  \centering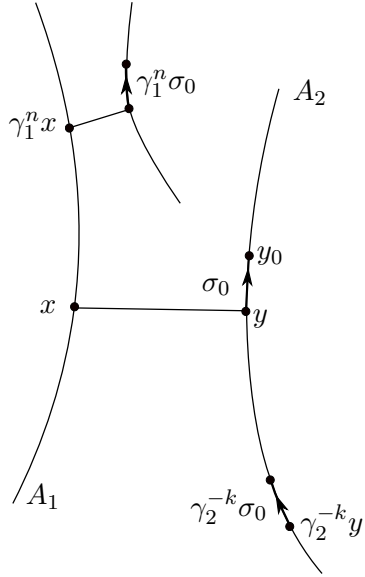 \\
	\caption{Location of axes}\label{fig1}
\end{figure}

Consider the following path,  $\frakp$, which is obtained
 by concatenating quasigeodesic segments: $$\frakp= [\gamma_2^{-k}y,
 y] \cup [y,x] \cup [x, \gamma_1^n x] \cup [\gamma_1^n x, \gamma_1^n y ]\cup
 [\gamma_1^n y , \gamma_1^n y_0]. $$

Notice first that $\frakp$ is a $(1, 4\cdot
45\delta)$--quasi-geodesic. To prove this, it is sufficient to observe
that each of the Gromov products $(\gamma_2^{-k} y \cdot x
)_y$, $(y \cdot \gamma_1^n x)_x  $, $(x \cdot \gamma_1^n y )_{\gamma_1^n x}$, and $(
\gamma_1^n x \cdot \gamma_1^n y_0 )_{ \gamma_1^n y }$ are bounded above by
$20\delta$.  This fact follows from the definitions of the points (sometimes after translating
the configuration by $\gamma_1^{-1}$).

Now,  we claim that 
$\gamma_2^{-k}y$ is   
$300\delta$--quasi-minimal for $\gamma_1^n
\gamma_2^k$. 
To show this, we will use the following fact, which is standard.  If a point $p$ is not $300\delta$--quasi-minimal for an isometry $\phi$ of a
$\delta$--hyperbolic space, then for any point $p'\in [p,\phi p]$ at distance
$100\delta$ from $p$ 
one has $d(p', \phi p') \leq d(p, \phi p) - 190\delta$. 
(Take a
  $\delta$--quasi-minimal point $p_0$ closest to $p$.  Since $p$ is not
  $300\delta$--quasi-minimal, it is at distance at least
  $145\delta$ from $p_0$.  Therefore $p'$ is at
  least $97\delta$--closer to $p_0$ than to $p$,  which easily ensures the
  claim.)

Let $p=\gamma_2^{-k} y$ and 
 $\phi = \gamma_1^n\gamma_2^k$. 
Since a point $p'\in[p,\phi p]$ at distance $100\delta$ from $p$ must lie on $\gamma_2^{-k}\sigma_0$, it follows that $\phi p' \in \gamma_1^n\sigma_0$.  Since
 $\frakp$  is a quasigeodesic,  
$$|d(p', \phi p') - d(p, \phi p) | \leq 4\cdot 45 \delta =
180\delta.$$  Therefore, the above fact implies that 
$\gamma_2^{-k}y$ is  $300\delta$--quasi-minimal for $\gamma_1^n
\gamma_2^k$.

Since $\gamma_2^{-k}y$ is $300\delta$--quasi-minimal for $\phi$ and is displaced by more than $800\delta$, $\phi$ must be loxodromic.  Moreover,  the $120\delta$--quasi-axis $A=A_{120\delta}(\phi)$ 
contains $\gamma_2^{-k}\sigma_0$, and therefore its image $\gamma_1^n
\sigma_0$.  In other words, 
 the closest point projection of $A$ 
onto $A_2$ has non-empty intersection with $\gamma_2^{-k}\sigma_0$.
Thus, for  $k$ sufficiently large, the repelling point of   $\gamma_1^n\gamma_2^k$ is in
$U$. 

Let $\sigma_1$ be a subsegment of $A_1$ of length $1000\delta$ ending at $\gamma_1^nx$.  By applying the same argument to $\phi'=(\gamma_1^{n}\gamma_2^{k})^{-1}$ with $\sigma_1$ in place of $\sigma_0$, we conclude that for $n$ sufficiently large, the repelling point of $\phi'$ (which is the attracting point of $\gamma_1^n\gamma_2^k$) is in $V$.


\end{proof}


Here is the main proposition of this section.

\begin{prop} \label{gamma} Let $G$ be a group that admits a cobounded acylindrical action by isometries on a non-elementary geodesic hyperbolic space $(X,d_X)$.  Let
  $H_1,\dots,H_k\subseteq G$ be a family of subgroups that are elliptic for the action of $G$ on
  $X$ such that for every $i$, each non-trivial subgroup $H$ of $H_i$ satisfies $\Fix_{50\delta}(H) \neq X$.  Then
  there exist infinitely many independent loxodromic elements whose $10\delta$--quasi-axes have bounded intersection with each $\Fix_{50\delta}(H_i)$.

Moreover, among them, there are infinitely many independent loxodromic elements $\gamma$ such that no element in $\bigcup_i H_i\setminus\{1\}$  preserves the pair $\{\gamma^+, \gamma^-\}$ in $\partial X$.
\end{prop}

\begin{proof}
We begin by proving the first assertion. Throughout the proof, let $F(H)=\Fix_{50\delta}(H)$.  If $H_i$ is an
infinite order elliptic subgroup, then $F(H_i)$ is a bounded set, and
the first assertion holds for every loxodromic element.  Thus it suffices
to assume that $H_1,\dots,H_k$ have finite order.  In this case, we
will show that we can find infinitely many loxodromic elements
$\gamma$ such that $\gamma^\pm\not\in \overline{F(H_i)}$ for all $i$.
It will then follow that $\diam(A_{10\delta}(\gamma)\cap F(H_i))$ is bounded for all $i$.

Observe that since  $\Fix_{50\delta}(H_i) \neq X$ for each $i$,  and
$X$ is non-elementary, we
have that $\overline{F(H_i)} \neq \partial X$. 

We will proceed by induction.  Choose any loxodromic element, $g_1\in
G$.  If $g_1^\pm\not\in \overline{F(H_i)}$ for all $i$, we stop.  Otherwise, choose another loxodromic element $g_2\in G$ with
$g_2^\pm\not\in\overline{F(H_1)}$. Such an element can be found since $\overline{F(H_1)}$ is
not the entire boundary and the action of $G$ on $X$ is cobounded.  If $g_2^\pm\not\in\overline{F(H_i)}$ for all $i$, we stop.  Otherwise, let $U\subset \partial X$ be an open set containing $g_2^+$ such that $U\cap\overline{F(H_1)}=\emptyset$.  By Lemma \ref{LOA}, there exists $N_0$ such that for all $n_0>N_0$,  $(g_2^{n_0}g_1^{n_0})^+\in U$.  Letting $\gamma_0=g_2^{n_0}g_1^{n_0}$, it follows that $\gamma_0^+\not\in \overline{F(H_1)}$.  By Lemma \ref{lem:axis}, it follows that, additionally, $\gamma_0^-\not\in\overline{F(H_1)}$.  If $\gamma_0^\pm\not\in \overline{F(H_i)}$ for all $i$, we stop. 

Otherwise, as $\overline{F(H_2)}$ is not the entire boundary and the action of $G$ on $X$ is cobounded,  we may choose a   a loxodromic element
$g_3$ with $g_3^\pm\not\in\overline{F(H_2)}$.  If $g_3^\pm\not\in\overline{F(H_i)}$ for all $i$, we stop.  
Otherwise, let $U_1\subset \partial X$ be an open subset containing $g_3^-$ such that $U_1\cap \overline{F(H_2)}=\emptyset$.
By Lemma \ref{LOA}, there exists $N_1$ such that for all $n_1>N_1$, $(\gamma_0^{n_1} g_3^{n_1})^+\in U$ and $(\gamma_0^{n_1} g_3^{n_1})^-\in U_1$.  Let $\gamma_1=\gamma_0^{n_1} g_3^{n_1}$.  Then Lemma \ref{lem:axis} implies that $$\gamma_1^\pm\not\in \overline{F(H_1)}\cup\overline{F(H_2)}.$$    If $\gamma_1^\pm\not\in\overline{F(H_i)}$ for all $i$, we stop.

Otherwise, we continue this procedure, and after at most $k-1$ iterations we will have produced an element $\gamma=\gamma_j$ for some $0\leq j\leq k-1$ such that $\gamma^\pm\not\in\overline{F(H_i)}$ for all $i$.  To find infinitely many such elements $\gamma$, we need only change the final step of the procedure.  By construction, $\gamma=\gamma_{j-1}^{n_{j}}g_{j+2}^{n_j}$ where $n_j>N_j$ for $N_j$ provided by Lemma \ref{LOA}.  However, it is clear from the construction that the exponents of $\gamma_{j-1}$ and $g_{j+2}$ may be chosen independently, so for each $n_j,n_j'>N_j$, $$(\gamma_{j-1}^{n_j}g_{j+2}^{n_j'})^\pm\not\in\overline{F(H_i)} \quad \textrm{for all $i$}.$$  

If this procedure terminated early, then we chose a loxodromic element $g=g_j$ for some $0\leq j<k-2$ such that $g^\pm\not\in \overline{F(H_i)}$ for all $i$.  In this case, we proceed as follows.  Choose any $h\in G$, and consider $h^{-1}gh$, a loxodromic element such that $A_{10\delta}(h^{-1}gh)\cap A_{10\delta}(g)=\emptyset$.    Let $U$ be an open set containing $g^+$ such that $U\cap \left(\bigcup_{i=1}^k \overline{F(H_i)}\right)=\emptyset$.  By applying Lemma \ref{LOA} first to the pair $g$ and $h^{-1}gh$ and then again to the resulting element $g^nh^{-1}g^nh$ and $g^{-1}$, we can conclude that there exists an $N\geq 0$ such that for each $n\geq N$, $((g^nh^{-1}g^nh)^ng^{-n})^\pm\in U$.  Thus we have constructed infinitely many loxodromic elements $\gamma=(g^nh^{-1}g^nh)^ng^{-n}$ with $\gamma^\pm\not\in\overline{F(H_i)}$ for all $i$, as desired.

We now prove the second assertion. Consider an element $\gamma$ found by the above procedure. Then the point $\gamma^+ \in \partial X$ is outside $\bigcup_i \overline{F(H_i)} \cap \partial X$, which is closed in $\partial X$. There is a neighborhood $U_{\gamma^+}$ containing $\gamma^+$ which avoids  $\bigcup_i \overline{F(H_i)} \cap \partial X$, and, given an arbitrary $R>100\delta$,  we may  choose this neighborhood so that its convex hull in $X\cup \partial X$ (i.e., the union of geodesics joining its points) is at distance at least $R$ from any $F(H_i)$. The action of $G$ on $X$ is cobounded,  so for all $\xi \in U_{\gamma^+}$, 
 there is a sequence of conjugates of $\gamma$, which we will call $\gamma_n$, such that  $(\gamma_n)^+$ and   $(\gamma_n)^-$ converge to $\xi$. If $h\in H_i$ preserves the pair $\{\gamma_n^+,\gamma_n^-\}$, then $F_{20\delta}(h)$ must contain the segment between $F(H_i)$ and its closest point projection on the quasi-axis $A_{10\delta}(\gamma_n)$. Note that this segment has a long initial subsegment that is $2\delta$-close to any geodesic ray from $F(H_i)$ to $\xi$.  With a suitable choice of $R$, the acylindricity of the action implies that only finitely many elements can  $20\delta$-quasi-fix  such a segment.  By extracting an appropriate subsequence, one concludes that if there is such a non-trivial element for each $\gamma_n$ 
 then there exists  $i$ and an element $h_i \in H_i\setminus\{1\}$ fixing an open neighborhood inside $U_{\gamma^+} $, say $U'$. However, there is a loxodromic element in $G$ such that both fixed points at infinity are contained in $U'$, and a large power of this element, which we will call $g'$, sends $\partial X\setminus U'$ inside $U'$.  Thus $\overline{F(h_i)}$ and $\overline{F(h_i^{g'})}$ cover $\partial X$. This contradicts the first point of this proposition.

\end{proof}




\section{Ping-pong with everyone}

\begin{prop}\label{pingpong}
Let $G$ be a group acting  acylindrically by isometries on a geodesic $\delta$--hyperbolic space $(X,d_X)$, and let $G_0$  be a subgroup of $G$ that is elliptic for the action on $X$. Assume that
$\gamma \in G$ is loxodromic on $X$ such that no element of $G_0\setminus\{1\}$ preserves  $A_{10\delta}(\gamma)$, and such that there is a constant $D\geq 0$ satisfying  
$\diam(\Fix_{50\delta}(G_0)\cap A_{10\delta}(\gamma))\leq D$.  Then there is a constant $C\geq 0$ such that for any $N$ such that the translation length of $\gamma^N$ is at
least $C$, the group generated by $\gamma^N$ and $G_0$ is
the free product of $\langle \gamma^N \rangle$ and $ G_0 $. 
\end{prop}

\begin{proof}


Let $D'$ be an upper bound on the diameter of the closest point projection of the
$10\delta$--quasi-axis of one conjugate of $\gamma$ to the $10\delta$--quasi-axis of a conjugate with different fixed points at infinity. As $X$ is a geodesic $\delta$--hyperbolic space, there exist constants $K_1,K_2$ such that all $K_1$--local $(1,10\delta)$--quasigeodesics are $K_2$--global quasigeodesics.  Let
$\Delta = \max\{D,D', 1000\delta, K_1,K_2^2\}$.  Choose $N$ such that the translation length of $\gamma^N$ is at
least $ 10 \Delta +1000\delta$.    

Consider the free product $H= G_0 *\langle t\rangle$. We will show that the natural
homomorphism  from
 $H$ to $\langle \gamma^N, G_0\rangle$ sending $G_0$ to $G_0$ and $t$ to 
$\gamma^N$ is injective. It will then follow that it is an isomorphism.

In order to show injectivity, consider a normal form for an
element of $H$:
$$ h =   g_0^{(1)} t^{l_1} g_0^{(2)} t^{l_2} \dots g_0^{(k)} t^{l_k}, $$
with $ l_j \in \mathbb{Z}\setminus\{0\}$ for $1\leq j<k$, $l_k\in \mathbb Z$, $g_0^{(i)} \in
G_0\setminus\{1\}$ for $1<i\leq k$, and $g_0^{(1)}\in G_0$.  The proof will only depend on the conjugacy
class of $h$, and thus we may assume $g_0^{(1)}\neq 1$.  In what follows, we identify $\langle t\rangle$ with its image $\langle \gamma^N\rangle$ in $G$.

Pick a base point $x_0\in{\rm Fix}_{10\delta}(G_0)$ and write 

$$x_{2j-1} = g_0^{(1)} t^{l_1}  \dots g_0^{(j)} x_0,   $$ 

$$x_{2j} = g_0^{(1)} t^{l_1}  \dots g_0^{(j)} t^{l_j}  x_0. $$ 

Let us also write $h_{2j-1} = g_0^{(1)} t^{l_1}  \dots g_0^{(j)}$.

Consider geodesics
$\frakp_i = [x_i, x_{i+1}]$. Thus for each
even $i$, $\frakp_i $ is a geodesic of length at most $10\delta$, and
for each odd $i$, $\frakp_i $ is a geodesic containing a central
subsegment of length at least $10 \Delta+ 950\delta $  that is
in     
 the $10\delta$--quasi-axis of $\gamma^{ h_{i}^{-1}} $.  

When $i$ is odd, let $R_i$ be the distance from $x_i$ to the
quasi-axis of $ \gamma^{h_{i}^{-1}} $.  Observe that $R_i \leq  | \frakp_i  | /2   -
5\Delta -400\delta$ since the translation length of $ \gamma^{h_{i}^{-1}} $ is at
least $10\Delta +1000\delta$.  Observe also  that $|R_i-R_{i+2}|\leq 10\delta$, since 
  $(h_i t^{l_{(i+1)/2}} )^{-1}h_{i+2}$ 
  is an element of $G_0$, and hence moves $x_0$ by at
  most $10\delta$.

 \begin{figure}[h]
\def\svgwidth{3in}  
  \centering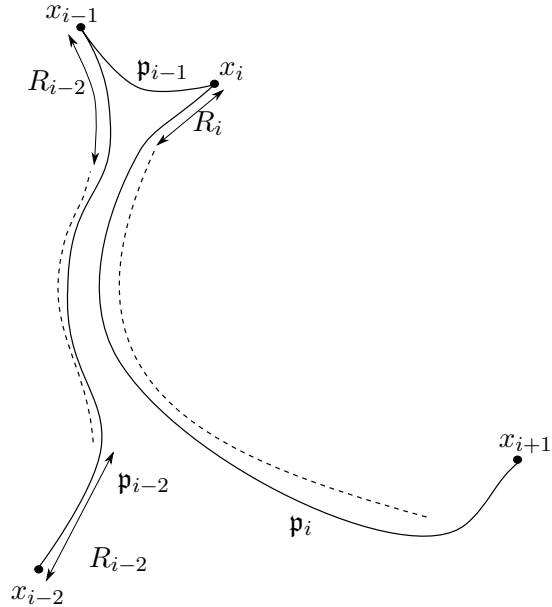 \\
	\caption{If the Gromov product $( x_{i+1} \cdot  x_{i-2})_{
  x_{i}}$ is larger than  $R_i+2\Delta +100\delta$ for $i$ odd, then the quasiaxes of $\gamma^{h_{i-2}^{-1} } $ (dotted, left) and of
$\gamma^{  h_{i }^{-1}}$ (dotted, right) remain
$10\delta$--close  for a length at least 
       $  2\Delta  $. 
 }\label{fig2}
\end{figure}


We now find an upper bound on the Gromov product  $( x_{i+1} \cdot  x_{i-2})_{
  x_{i}}$ for $i$ odd. Assume towards a contradiction that it is larger than  $R_i+2\Delta +100\delta$. 
Then it follows that $\frakp_{i-2}$ and $\frakp_{i}$
remain $2\delta$--close to each other for a segment of length at least
$R_i+2\Delta +50\delta   $ (see Figure \ref{fig2}).    In particular, by the definition of $R_i$,   the quasi-axes of $\gamma^{h_{i-2}^{-1} } $ and of
$\gamma^{  h_{i }^{-1}}$ remain
$10\delta$--close  for a length of at least 
       $  2\Delta  $.

 \begin{figure}[h]
\def\svgwidth{3in}  
  \centering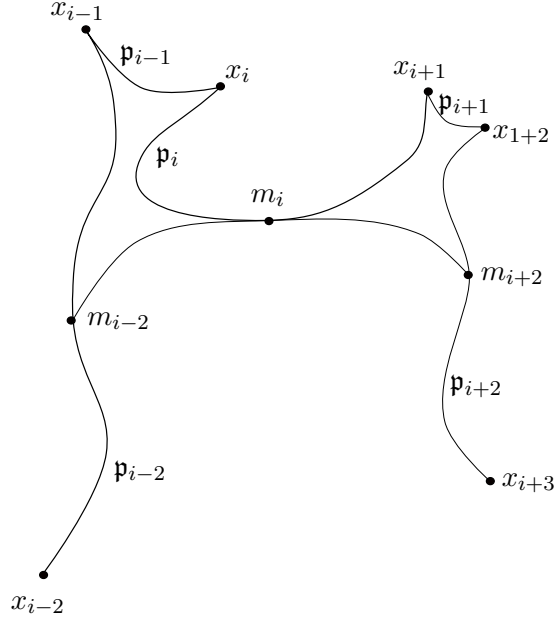 \\
	\caption{The concatenation of segments $[m_{i},m_{i+2}]$, for odd $i$, is a quasigeodesic.}\label{fig3}
\end{figure}


By the definition of $\Delta$,
this implies that their fixed points at infinity are the same, and thus $g_0^{((i-1)/2)}$ preserves $\{\gamma^+, \gamma^-\}$. But this is forbidden by assumption, and we have our contradiction.


Therefore,   $$( x_{i+1} \cdot  x_{i-2})_{ x_{i}}   \leq | \frakp_i  | /2   - 2\Delta +
100\delta,   $$ and, similarly,  $$( x_{i+1} \cdot  x_{i-2})_{ x_{i}}\leq | \frakp_{i-2}  | /2
- 2\Delta + 100\delta.$$

Now let $m_{2i+1}$ to be the midpoint of $\frakp_{2i+1}$ for all $0\leq i\leq k-1$, let
$m_{-1}=1$, and let $m_{2k+1} = h$. 
The final segment of  $[m_{2i-1}, m_{2i+1}] $ of length
$2\Delta$  and the initial segment of $[m_{2i+1}, m_{2i+3}] $ of
length $2\Delta$  are $2\delta$-close to $\frakp_{2i+1}$, and  $( m_{2i-1} \cdot  m_{2i+3})_{ m_{2i+1}}   \leq 2\delta$. Thus, the concatenation of the geodesic segments $[m_{2i-1}, m_{2i+1}]$ are a $(2\Delta)$--local
$(1,10\delta)$--quasigeodesics.  
By the choice of $\Delta$, this concatenation is therefore a global
$K_2$--quasigeodesic (see Figure \ref{fig3}.)  The length of this quasigeodesic is more than $4\Delta$, and thus it follows that its
endpoints are at a distance greater than $4\Delta/K_2 - K_2\geq 0$.  Setting $C=10\Delta+1000\delta$ completes the proof.

\end{proof}

\begin{theo}\label{theo;pingpong_with_elliptic}
 Let $G$ be a group with no non-trivial
finite normal subgroups.  Let $(X,d)$ be a non-elementary geodesic  $\delta$--hyperbolic space upon which
$G$ acts acylindrically, coboundedly, and by isometries. Then for every finite collection
$\{H_1,\dots, H_k\}$  of subgroups
of $G$ that are elliptic for the action on $X$, there exists an element $\gamma\in
G$ such that $\langle \gamma, H_i\rangle = \langle
\gamma\rangle * H_i$ for all $1\leq i\leq k$. 

\end{theo}

\begin{proof}
We first prove that      for all   elliptic subgroups $H$, ${\rm
  Fix}_{50\delta}(H) \neq X$.   
Assume that there is a subgroup $H$ which is elliptic for the action on $X$ such that ${\rm
  Fix}_{50\delta}(H) = X$.   By acylindricity of the action and non-elementarity of $X$,   $H$ must be
finite. But any conjugate $H'$ has also the property that ${\rm
  Fix}_{50\delta}(H') = X$, and acylindricity then implies that the number
of conjugates of $H$ is finite.  It follows that the normalizer $N_G(H)$ of $H$ has
finite index in $G$.  There exists a finite-index normal subgroup $G_0$ of
$G$ contained in  $N_G(H)$ (for instance the intersection
of the conjugates of $N_G(H)$).  Thus $H$ and all of its conjugates are normal in $G_0$, since they are images of $H$ by
automorphisms of $G_0$. Since there are finitely many such conjugates,
their product is still a finite normal subgroup of  $G_0$, and it is
also normal in $G$.  Hence $G$ has a
non-trivial finite normal subgroup, which is contrary to our
assumptions.

Therefore, for all   elliptic subgroups $H$, ${\rm
  Fix}_{50\delta}(H) \neq X$.   Thus by  Proposition \ref{gamma},
there exists a loxodromic element $\gamma\in G$ and a constant $D\geq 0$ such
that $A_{10\delta} (\gamma)\cap \Fix_{50\delta}(H_i)$ has diameter at most $D$ and no element of $H_i$ preserves $\{\gamma^+,\gamma^-\}$. 

Finally, it follows from Proposition \ref{pingpong} that for each $i$ there is a constant $N_i$ such that for all $N\geq N_i$, $\langle \gamma^{N},H_i\rangle\simeq\langle \gamma^{N}\rangle *  H_i$.  Setting $N=\max\{N_1,\dots, N_k\}$, the result follows.
\end{proof}

The previous theorem applies to a wide variety of situations, for
instance a collection of quasi-convex subgroups of a hyperbolic
group, a collection of relatively quasi-convex subgroups of a relatively
hyperbolic group, or a hyperbolically embedded collection of subgroups of an
acylindrically hyperbolic group,
with respect to some generating set (in the sense of \cite[Definition 4.25]{DGO}; see below for the precise definition).


We illustrate two particularly nice instances of these situations.  The first is property $P_{naive}$.

\begin{theo} \label{theo;main}  
Let $G$ be an acylindrically hyperbolic group with no non-trivial
finite normal subgroup.  Then $G$ has property $P_{naive}$.
\end{theo}

We first need the following lemma. In order to state it, we recall the process of coning-off a subset of a hyperbolic space. Following \cite[\S 5.3]{DGO}, given a collection  $\mathcal{Q}$  of $10\delta$-quasi-convex subsets of a hyperbolic space $X$, we define the quantity $\Delta(\mathcal{Q})$ to be the supremum over $Q_1 \neq Q_2$ in   $\mathcal{Q}$  of the diameter of $Q_1^{+20\delta} \cap Q_2^{+20\delta} $, where $Q^{+d}$ is the set of points at distance at most $d$ from $Q$.  Given a hyperbolic space with a family $\mathcal Q$ of quasi-convex subsets such that $\Delta(\mathcal{Q})$ is finite, then, after possible rescaling, one can apply the cone-off theorem \cite[Corollary 5.39]{DGO}.  This states that equivariantly gluing  hyperbolic cones of bounded diameter on each $Q\in \mathcal{Q}$ produces a hyperbolic space, which we refer to as the coned-off space.  Moreover, if the initial action is acylindrical, \cite[Prop. 5.40]{DGO} guarantees that the action on the coned-off space is acylindrical as well.

\begin{lemma} \label{elliptic}
Let $G$ be a group acting acylindrically 
 by isometries on a geodesic hyperbolic space $(X,d_X)$, and let $g_0$  be an element of $G$. Then there exists a
geodesic hyperbolic space $(Y,d_Y)$ that admits a cobounded acylindrical action of $G$ by isometries in which  $g_0$ is elliptic.

Moreover, if $G_1, \dots, G_k$  are subgroups of $G$ that are quasi-convex in $G$
with respect to the metric $d_X$, and whose set $\mathcal{Q}$ of cosets has  $\Delta(\mathcal{Q}) <\infty$ in $X$ 
 then there   exists a
geodesic hyperbolic space $(Y,d_Y)$ that admits a cobounded acylindrical action of $G$ by isometries in which each  $G_i$ is elliptic.
\end{lemma}

\begin{proof} Let us prove the first part of the statement. By acylindricity of the action of $G$ on $X$, $g_0$ is either elliptic or loxodromic (see for
instance \cite[Lemma 2.2]{Bo}). 
If $g_0$ is elliptic for the action on $X$, there is nothing to prove.
 Hence we assume it is
loxodromic.  Then the collection of its quasi-axes is a  collection $\mathcal{Q}_0$ of $10\delta$-quasi-convex subsets of $X$ such that $\Delta (\mathcal{Q}_0)$ is finite; this follows from  acylindricity and \cite[Proposition
6.29]{DGO}.       Thus  \cite[Theorem
5.39, Prop. 5.40]{DGO} provides a suitable space.

For the second part, we directly apply \cite[Theorem
5.39, Proposition 5.40]{DGO}. 
\end{proof}

\begin{proof}[Proof of Theorem \ref{theo;main}]
 Let $h_1,\dots, h_k\in G$ be any elements of $G$.  As $G$ is
 acylindrically hyperbolic, it acts acylindrically, coboundedly, and by isometries on a geodesic hyperbolic space
 $(X,d_X)$.  By repeatedly applying Lemma \ref{elliptic}, we may
 assume that there is a geodesic metric space $(Y,d_Y)$ with an
 isometric $G$--action that is 
 acylindrical in which $h_1,\dots,h_k$ are all elliptic.     Applying Theorem \ref{theo;pingpong_with_elliptic} completes the proof.
\end{proof}

The second instance deals with a collection of hyperbolically embedded subgroups.  We briefly recall the definition here and refer the reader to \cite{DGO} and \cite{O} for further details.

Given a collection of subgroups $\{H_1,\dots, H_n\}$ of $G$, let $\mathcal H=\sqcup_{i=1}^n H_i$.  If $X$ is relative generating set (i.e., a subset $X\subseteq G$ such that $X\cup H_1\cup\cdots\cup H_n$ generates $G$), we consider the Cayley graphs $\Gamma(G,X\sqcup\mathcal H)$ and $\Gamma(H_i,H_i)$, and we naturally think of the latter as subgraphs of the former.  For each $i$, we define a relative metric $\widehat d_i\colon H_i\times H_i\to[0,\infty]$ as follows.  For all $h,k\in H_i$, $\widehat d_i(h,k)$ is defined to be the length of the shortest path in $\Gamma(G,X\sqcup \mathcal H)$ that connects $h$ to $k$ and does not contain any edges of $\Gamma(H_i,H_i)$.  

A collection of subgroups $\{H_1,\dots, H_n\}$ of $G$ is \emph{hyperbolically embedded  in $G$ with respect to a subset $X\subseteq G$} 
if the following conditions hold.
\begin{enumerate}
\item[(a)] The group $G$ is generated by $X$ together with the union of all $H_i$, and the Cayley graph $\Gamma(G,X\sqcup\mathcal H) $ is hyperbolic.
\item[(b)] For every $i$, the metric space $(H_i,\widehat d_i)$ is proper, i.e., every ball (of finite radius) in $H_i$ with respect to the metric $\widehat d_i$ contains finitely many elements.
\end{enumerate}

If there exists a relative generating set $X$ such that $\{H_1,\dots, H_n\}$ hyperbolically embeds into $(G,X)$, then we say $\{H_1,\dots, H_n\}$ \emph{hyperbolically embeds into $G$}.

\begin{theo}\label{theo;pingpong_he}
 Let $G$ be a group with no non-trivial
finite normal subgroup, and let $\{H_1,\dots, H_n\}$ be a collection of subgroups of $G$ which hyperbolically embeds into $(G,Y)$ for some relative generating set $Y\subset G$.  
Then there exists an element $\gamma\in
G$ such that  for all $k=1,\dots,n$, $\langle \gamma, H_k\rangle = \langle
\gamma\rangle * H_k$. 

\end{theo}

\begin{proof}
Let $\mathcal H=H_1\sqcup \cdots\sqcup H_n$. By \cite[Theorem 5.4]{O}, the action of $G$ on the Cayley graph $\Gamma(G,Y\sqcup \mathcal H)$ is nonelementary and acylindrical.  Moreover, the action is clearly cobounded, and each $H_i$ is elliptic with respect to this action.  Thus Theorem \ref{theo;pingpong_with_elliptic} establishes the result.  

\end{proof}

 It is not clear when one can replace the assumption that the collection
  of groups is  hyperbolically embedded as a collection by the weaker
  assumption that each group is hyperbolically embedded. If  one has a collection of subgroups of an
  acylindrically hyperbolic group $G$, each of which is elliptic in a non-elementary
  acylindrical action of $G$,  
   when is there an acylindrical action in which all are elliptic
   simultaneously? When can one find an element that plays ping-pong with
   each of these groups?
   
\section{Primitive group rings}

We now turn our attention to the group ring of an acylindrically hyperbolic group.  Recall that a ring $R$ is \emph{(right) primitive} if it has a faithful irreducible (right) $R$--module.

\begin{prop} \label{groupring}
Let $G$ be a countable acylindrically hyperbolic group with no non-trivial finite normal subgroup.  Then for any countable domain $R$, the group ring $RG$ is primitive.
\end{prop}

Before beginning the proof of Proposition \ref{groupring}, we recall the following criterion of Alexander and Nishinaka \cite{AN}.

\begin{quote} $(*)$ For each subset $M$ of $G$ consisting of a finite number of elements not
equal to $1$, and for any positive integer $m \geq 2$, there exist distinct $a, b, c \in G$ so that if $(x^{-1}_1g_1x_1)(x^{-1}_2g_2x_2)· · ·(x^{-1}_m g_mx_m) = 1,$ where $g_i \in M$ and
$x_i \in \{a, b, c\}$ for all $i = 1, . . . , m$, then $x_i = x_{i+1}$ for some $i$.
\end{quote}

\begin{thm}\label{AN} \cite{AN}
Let $G$ be a group which has a non-abelian free
subgroup whose cardinality is the same as that of $G$, and suppose that $G$ satisfies
property $(*)$. Then if $R$ is a domain with $|R| \leq |G|$, the group ring $RG$ of $G$ over
$R$ is primitive. In particular, the group algebra $KG$ is primitive for any field $K$.
\end{thm}

We will need the following lemma.  An element $g\in G$ is \emph{generalized loxodromic} if $G$ admits an acylindrical action on a hyperbolic space with respect to which $g$ is loxodromic.   Given a generalized loxodromic element $g$, we denote by $E(g)$ the unique maximal virtually cyclic subgroup containing $g$.

\begin{lem}\label{noloops}
Let $G$ be an acylindrically hyperbolic group, $u\in G$ a generalized loxodromic element, and $g_1,\dots, g_k\in G$ elements such that $g_i\not\in E(u)$ for $i=1,\dots,k$.  Then there exists a constant $N\geq 0$ such that  
\[u^{n_0}g_1u^{n_1}g_2\dots u^{n_{k-1}}g_ku^{n_k}\neq1\] 
if $|n_i|\geq N$ for all $i$.
\end{lem}

\begin{proof}
Since $u$ is generalized loxodromic, it follows from \cite[Corollary 2.9]{DGO} that there is a subset $X\subset G$ such that  $E(u)$ hyperbolically embeds into $G$ with respect to $X$.  By \cite[Corollary 4.27]{DGO}, we may assume that the finite set $\{u,g_1,\dots, g_k\}\subseteq X$.  By a special case of \cite[Lemma 4.11(b)]{DGO}, there is a finite subset $Z\subset E(u)$ such that $d_Z$ is Lipschitz equivalent to $\widehat d$ (where multiplication and division are extended to include $\infty$ in the natural way).  As $u\in X$, it follows that $\widehat d$ (and therefore $d_Z$) is finite for every pair of elements in $E(g)$, and so $Z$ generates $E(u)$.  Since, in addition, $Z$ is finite, there exists a $K\geq 0$ such that for any $i$, we have that $i=d_{\{u\}}(1,u^i)\leq K d_Z(1,u_i)$.  Combining these facts with \cite[Lemma 4.14]{DGO} yields that there is a constant $N\geq 0$ such that for any $n\geq N$, if $u^n$ is the label of an $E(u)$--component $a$ in a $(2k)$-gon $P$, then $a$ is not isolated in $P$.

Let $n_i\geq N$ for all $i=1,\dots, k$, and suppose \[w=u^{n_0}g_1u^{n_1}g_2\dots u^{n_{k-1}}g_ku^{n_k}=1.\]  Let $P$ be a $(2k)$--gon in $\Gamma(G,E(u)\sqcup X)$ with label $w$ such that  even numbered sides $a_i$ of $P$ are edges labeled by $g_i\in X$ for $1\leq i\leq k$ and odd numbered sides $b_j$ of $P$ are geodesics labeled by powers of $u$.  By assumption, $g_i\not\in E(u)$ for any $i$, and thus the $E(u)$--components of $P$ are precisely the odd numbered sides $b_j$.  Moreover, 
since $n_i\geq N$ for all $i$, no $b_j$ can be isolated in $P$.   

If there exists some $j$ such that $b_j$ is connected to $b_{j+1}$, then there is an edge labeled by an element $e$ of $E(u)$ connecting $(b_j)_+$ to $(b_{j+1})_-$.  Therefore $g_j\in E(u)$, which contradicts our assumption on $u$.

Choose $i_1\in\{2,\dots, k\}$ minimal such that $b_1$ connects to $b_{i_1}$.  If $i_1=2$ or $k$, then we have a contradiction, as above.  Otherwise, let $e_1$ be an edge labeled by an element of $E(u)$ connecting $(b_{i_1})_-$ to $(b_1)_+$, and let $P_1$ be the polygon $e_1\cup a_1\cup b_2\cup \cdots \cup b_{i_1-1}\cup a_{i_1-1}$.  Since $P_1$ has at most $2k$ sides,  another application of \cite[Lemma 4.14]{DGO} yields that $b_2$ cannot be isolated in $P_1$.  Moreover, $b_2$ cannot connect to $e$, since this would imply that $b_2$ also connects to $b_1$, which is not possible by the discussion above.  Thus we may choose $i_2\in\{3,\dots, i_1-1\}$ minimal such that $b_2$ connects to $b_{i_2}$.  As there are only finitely many $E(u)$--components, this process has to terminate, and at the final step we either have some $b_j$ that is connected to $b_{j+1}$ or we have some $b_j$ which is isolated.  In either case we reach a contradiction, and we conclude that $w\neq 1$, as desired.

\end{proof}

The proof of Proposition \ref{groupring} now follows along the same lines as Solie's proof for non-elementary torsion-free hyperbolic groups in \cite{Sol}.

\begin{proof}[Proof of Proposition \ref{groupring}]
Let $\mathcal M$ be a finite collection of non-trivial elements of $G$.  By Theorem \ref{theo;main}, there exists an element $u\in G$ such that $\langle u,g\rangle= \langle u\rangle *\langle g\rangle$ for every $g\in\mathcal M$.  Note that by construction $u$ is a generalized loxodromic element of $G$, and the condition that $u$ and $g$ generate a free group implies that $g\not\in E(u)$ for all $g\in \mathcal M$.  

For each finite subset $A$ of $\mathcal M$, Lemma \ref{noloops} provides a constant $N(A)$.  As there are only finitely many finite subsets of $M$, let $N=\max\{N(A)\mid A\subset \mathcal M\}$.

Let $g_1,\dots, g_m\in \mathcal M$, and consider a word 
\[
w=(x_1^{_1}g_1x_1)(x_2^{-1}g_2x_2)\cdots (x_m^{-1}g_mx_m),
\]
where $x_i\in \{u^N,u^{2N},u^{3N}\}$ for each $i=1,\dots,m$.   Then 
\[w=u^{n_0}g_1u^{n_1}g_2\dots u^{n_{m-1}}g_mu^{n_m},
\]
where $u^{n_0}=x_1^{-1}$, $u^{n_i}=x^ix_{i+1}^{-1}$ for $i=1,\dots, m-1$, and $u^{n_m}=x_m$.  By construction, $n_i\in \{0,\pm N, \pm 2N\}$ for each $i=1,\dots, m-1$, and $n_1,n_m\neq 0$.  If $w=1$, then by Lemma \ref{noloops} some $n_i=0$, and it must be that $i\in \{1,\dots, m-1\}$.  Thus $x_ix_{i+1}^{-1}=1$.  Letting $a=u^N$, $b=u^{2N}$, and $c=u^{3N}$, and noting that that these are distinct elements since $u$ is an infinite order element, we conclude that $G$ satisfies property $(*)$.

Finally, $G$ has non-abelian free subgroups $F$ of the same cardinality as $G$ by \cite[Theorem 8.7]{DGO}.  Applying Theorem \ref{AN} completes the proof.

\end{proof}

\section{Boundaries}

Recall that a discrete group $G$ is said to be $C^*$--simple if its reduced $C^*$-algebra (the norm closure of the algebra of operators on $\ell^2(G)$)  is simple as a normed algebra.

By establishing property $P_{naive}$, we have thus obtained a new proof of the $C^*$-simplicity of acylindrically hyperbolic groups without non-trivial finite subgroups, by \cite[Lemmas 2.1 and 2.2]{BCdlH}.

In a recent work,  Kalantar and Kennedy  have characterized $C^*$-simplicity of a group in terms of actions on boundaries \cite{KK}. It is interesting to compare what we did with this language. Let us follow this point of view. Let $G$ be a finitely generated group, and let $B$ be a topological space endowed with an action of $G$ by homeomorphisms. We say that the action is \emph{minimal} if every $G$--orbit in $B$ is dense in $B$. We say it is \emph{proximal} if for all $x,y\in B$ there is a sequence $(t_n)$ of elements of $G$ such that $\lim t_n x = \lim t_n y$. It is \emph{strongly proximal} if the action of $G$ on $Prob(B)$, the space of probability measures on $B$, is proximal.

We say that $B$ is a \emph{$G$--boundary} if $B$ is compact 
 and the $G$--action is minimal and strongly proximal. In  \cite[Thm 1.5]{KK},  Kalantar and Kennedy   proved that $G$ is $C^*$--simple if and only if there exists a $G$--boundary on which the $G$--action is topologically free in the following sense: for all $g\in G\setminus\{1\}$ the fix point set of $g$ in $B$ has empty interior  (see also \cite[Thm 3.1]{BKKO}).

By \cite[Thm 1.5]{KK} (and \cite[Lemma 3.3]{BKKO}), we also know that if $G$ is $C^*$--simple, the Furstenberg boundary of $G$ is a $G$--boundary on which the action is free. However, it often happens that there is another (more) interesting $G$--boundary coming from the geometric situation, on which the action is not free, but is topologically free. This is the case for torsion-free  hyperbolic groups  and their Gromov boundaries.

In our situation, let $G$ be an acylindrically hyperbolic group without non-trivial finite normal subgroup, and let $X$ be a hyperbolic space with a cobounded, acylindrical $G$--action by isometries.  Let $\partial X$ be the Gromov boundary of $X$.  In general, $\partial X$ is far from being compact, hence it is not a $G$-boundary.     However, our proof of property $P_{naive}$ reflects the following.

\begin{prop} \label{prop:boundaryaction}
Let $G$ be an acylindrically hyperbolic group without non-trivial finite normal subgroup, and let $X$ be a hyperbolic space with an acylindrical cobounded $G$--action by isometries.

The action of $G$ on   $\partial X$ by homeomorphisms is minimal, strongly proximal, and topologically free.
\end{prop}

\begin{proof}
Let $x_0\in X$. The shadow of a set $S\subset X$ is the subset of $\partial X$ consisting of points $\xi$ for which there exists $s\in S$ such that $( x_0 \cdot  \xi)_s \leq 5\delta$.   

Consider  two probability measures $\mu_1, \mu_2$ on $\partial X$. 
Since they have finite mass and there are uncountably many points in $\partial X$, there exists $\xi_0\in\partial X$ such that the following holds: for all $n$, there exists a neighborhood $U_n$ of $\xi_0$, for which $\mu_1(U_n) +\mu_2(U_n) <1/n$.  Indeed, this is easily seen by considering the negation of the statement.  We may take $(U_n)$ as a basis of neighborhoods of $\xi_0$.

On the other hand, take also $\zeta \in \partial X$ for which $(\zeta \cdot  \xi_0)_{x_0} \leq 5\delta$, and let $V_n$ be a basis of neighborhoods of $\zeta$.

By the coboundedness of the action, for each $n$ there exists $\gamma_n \in G$  loxodromic on $X$, with repelling fixed point in $U_n$ and there exists $\gamma'_n\in G$   loxodromic on $X$    with attracting fixed point in $V_n$.  Lemma \ref{LOA} provides constants $k,m$ such that $\gamma_n''={\gamma'_n}^m{\gamma_n}^k$ is a loxodromic element with repelling fixed point in $U_n$ and attracting fixed point in $V_n$.  

Fix $n$, and choose $m>n$  such that $C_n= \partial X\setminus U_n$ is outside $U_m$. 
We apply a large $M$-th power of $\gamma''_m$ so that  $(\gamma''_m)^MC_n$ is inside $V_m$. It follows that  $(\gamma''_m)^M \mu_i (V_m) \geq 1-\frac{1}{n}$. Letting $n$ go to infinity ensures that    $(\gamma''_m)^M \mu_i$ converges to the Dirac mass at $\zeta$.

This proves  minimality (taking $\mu_1 =\mu_2$ a Dirac mass) and strong proximality of the action of $G$ on $\partial X$.

Let us next prove topological freeness. First, if $\gamma$ is of infinite order in $G$, then its fixed point set is either a pair of points (if it is loxodromic), or empty (if it is not) by acylindricity (see \cite[Lemma 2.2]{Bo}).

If $\gamma$ has finite order and is non-trivial, assume by contradiction that $\zeta \in {\rm Fix}(\gamma)$ has an open neighborhood $U_0$ contained in in   ${\rm Fix}(\gamma)$. Then, by minimality, $U_0$ contains a fixed point of some loxodromic element $\gamma_0$ in $G$. By Lemma \ref{lem:axis} $U_0$ must then contain both fixed points at infinity   of $\gamma_0$, which we denote by  $\gamma_0^-, \gamma_0^+$.   By Proposition \ref{gamma} there also exists $\gamma_1 \in G$ which is loxodromic on $X$ such that $\gamma_1^+,\gamma_1^-\not\in  {\rm Fix}(\gamma)$.  
Considering $\gamma_0^n \gamma_1^m$ as in Lemma \ref{LOA}, we can extract  a sequence of elements $\xi_i\in G$ which are loxodromic on $X$  such that $\xi_i^+$ converges to $\gamma_0^+$ and  $\xi_i^-$  converges to $\gamma_1^-$. 
 It follows that eventually  $\xi_i^-$ is not in  ${\rm Fix}(\gamma)$, and therefore, by Lemma \ref{lem:axis}, neither is $\xi_i^+$.  However, $\xi_i^+$ must enter $U_0$ at some point, and we have a contradiction.


 \end{proof}

\end{document}